\documentclass[11pt]{amsart}

\usepackage[utf8]{inputenc}
\usepackage[T1]{fontenc}
\usepackage{amsmath,amssymb,amsthm,amsfonts}
\usepackage{mathtools}
\usepackage{enumitem}
\usepackage{hyperref}
\usepackage{geometry}
\usepackage{microtype}

\newtheorem{theorem}{Theorem}[section]
\newtheorem{proposition}[theorem]{Proposition}
\newtheorem{corollary}[theorem]{Corollary}
\newtheorem{lemma}[theorem]{Lemma}
\newtheorem{conjecture}[theorem]{Conjecture}

\theoremstyle{definition}

\newtheorem{example}[theorem]{Example}

\newcommand{\Aut}{\operatorname{Aut}}

\title[Abelian maximal subgroups of tame valued division algebras]
{Abelian maximal subgroups of valued division algebras}
\author{Huynh Viet Khanh}
\address{Department of Mathematics and Informatics, HCMC University of Education,
Ho Chi Minh City, Vietnam}

\email{khanhhv@hcmue.edu.vn}

\subjclass[2020]{16K20, 20E28}

\keywords{division algebras; multiplicative groups; maximal subgroups;
valued division algebras; tame division algebras}

\begin{document}
\begin{abstract}
Let $D$ be a division ring with center $K$. We study when the multiplicative
group $D^*$ can contain an abelian maximal subgroup. We prove that this
cannot occur for any noncommutative tame finite-dimensional central division
algebra over a Henselian valued field with nontrivial valuation. We also apply the same idea to standard valued division rings,
including twisted Laurent series, iterated Laurent series, and
Mal'cev--Neumann series. Finally, we introduce a complementary malnormality
obstruction and use it to rule out abelian maximal subgroups in quaternion
division algebras over real-closed fields.
\end{abstract}

\maketitle

\section{Introduction}

Let $D$ be a division ring with center $F$, and let $D^*$ denote its
multiplicative group.  The subgroup structure of $D^*$ is a basic and
difficult part of the theory of division rings.  Even the existence of maximal subgroups in $D^*$ is a difficult question.  In \cite{HaW1}, Hazrat and Wadsworth showed that if a
finite-dimensional division algebra has no maximal subgroup in its
multiplicative group, then the algebra and its center must satisfy very
restrictive conditions; in particular, the problem is reduced to the existence of
noncyclic division algebras of prime degree.  Their work also exhibits the
importance of valuation theory in this question: for suitable valued cyclic
division algebras they compute quotients of the form
$D^*/K^*(1+M_D)$ and use them to construct nonnormal maximal subgroups of
finite index.

The present paper is concerned with a more special, but still open, question:
can $D^*$ contain an abelian maximal subgroup when $D$ is noncommutative?
This question grew out of a series of papers on maximal subgroups of division
rings and skew linear groups.  Akbari, Mahdavi-Hezavehi, and Mahmudi studied in \cite{AMM} maximal subgroups of $\mathrm{GL}_1(D)=D^*$ and proposed, among other things, that nilpotent maximal subgroups should not occur in the multiplicative group of a
noncommutative division ring.  However, in  \cite{AEKG}, Akbari, Ebrahimian, Momenaei Kermani, and Salehi Golsefidy showed that the corresponding solvable statement is false: $\mathbb C^*\cup \mathbb C^*j$ is a nonabelian solvable maximal subgroup of the real quaternion division ring $\mathbb H^*$. They isolated the abelian case to formulate the following rigidity conjecture.

\begin{conjecture}
\label{conj:abelian-maximal}
Let $D$ be a division ring.  If $D^*$ contains an abelian maximal subgroup,
then $D$ is commutative.
\end{conjecture}

There is substantial evidence for this conjecture.  Ebrahimian proved in \cite{E} that nilpotent maximal subgroups of $D^*$ are abelian, thereby reducing the
nilpotent case to the abelian one.  In \cite{HT}, Hai and Thin proved centrality
results for locally nilpotent subnormal subgroups and studied locally
nilpotent maximal subgroups; related refinements were obtained previously by
Hai given in \cite{Hai}.  Later, Ramezan-Nassab and Kiani extended in \cite{RK1,RK2,RK3} several of these results to subnormal subgroups and to skew linear groups, including results on nilpotent, locally soluble, and polycyclic-by-finite maximal subgroups.  More recently, Khanh and Hai studied solvable-by-finite
and locally solvable maximal subgroups of almost subnormal subgroups of
$\mathrm{GL}_n(D)$ and of $D^*$, showing that the nonabelian locally solvable case is
forced into a cyclic prime-degree situation; see \cite{KH1,KH2}.  These results
all point to the same remaining obstruction: one must understand when the
multiplicative group of a maximal subfield can itself be maximal in $D^*$.

Our main result proves Conjecture~\ref{conj:abelian-maximal} for a large
class of valued division algebras.

\begin{theorem}
\label{thm:intro-main}
Let $K$ be a Henselian valued field with nontrivial valuation, and let $D$ be
a tame finite-dimensional central division algebra over $K$.  If $D$ is
noncommutative, then $D^*$ contains no abelian maximal subgroup.
\end{theorem}

The proof uses valuation theory in a way complementary to Hazrat--Wadsworth's
work \cite{HaW1}.  Their valuation-theoretic examples produce maximal subgroups by
computing quotients such as $D^*/K^*(1+M_D)$.  Here we use the principal unit
group $1+M_D$ in the opposite direction: it becomes part of an obstruction to
the existence of abelian maximal subgroups.

The group-theoretic observation is elementary.  If a group $G$ has a normal
subgroup $N$ such that both $N$ and $G/N$ are nonabelian, then $G$ has no
abelian maximal subgroup.  We say in this case that $G$ \textit{has property
$(\mathcal P)$}.  For a valued division ring $D$, the natural normal subgroup
is $1+M_D$.  We prove that if $D$ is finite-dimensional over its center and
the valuation on the center is nontrivial, then $1+M_D$ is nonabelian exactly
when $D$ is noncommutative.  On the quotient side, the canonical epimorphism
$D^*\to \mathbf{gr}(D)^*$ has kernel $1+M_D$.  Thus it remains to show that
$\mathbf{gr}(D)^*$ is nonabelian.  In the tame Henselian case this follows
from the standard characterization
\[
        [\mathbf{gr}(D):\mathbf{gr}(K)]=[D:K]
        \quad\text{and}\quad
        Z(\mathbf{gr}(D))=\mathbf{gr}(K).
\]
If $D$ is noncommutative, these two conditions force $\mathbf{gr}(D)$ to be
noncommutative.

Theorem~\ref{thm:intro-main} also excludes locally nilpotent maximal
subgroups in the same class.  Indeed, Khanh and Hai proved in \cite{KH2} that every locally nilpotent maximal subgroup of the multiplicative group of a division ring is
abelian.  Hence a noncommutative tame finite-dimensional central
division algebra over a nontrivially Henselian valued field has no locally
nilpotent maximal subgroup in its multiplicative group.

We include several examples to show the scope of the obstruction.  The valued
cyclic division algebras given in \cite{HaW1} provide a basic tame family.  We
also consider their valued quaternion example over a euclidean field, where
the residue division ring is $\mathbb H$.  Classical cyclic division algebras
of Brauer and Dickson, as presented in Lam \cite{L}, are treated by placing natural
valuations on them and passing, when needed, to Henselizations or completions.
Beyond the finite-dimensional tame setting, the same mechanism applies to
twisted Laurent series, iterated Laurent series, and Mal'cev--Neumann division
rings.

Finally, we explain why the valuative obstruction is not the whole story.
Hamilton's quaternion division ring $\mathbb H$ does not fit the property
$(\mathcal P)$ mechanism, but $\mathbb H^*$ still has no abelian maximal
subgroup.  This leads to a second obstruction.  If $A$ is an abelian maximal
subgroup of $D^*$ and $D$ is finite-dimensional over its center $F$, then
$L=A\cup\{0\}$ is a maximal subfield of $D$, and $L^*/F^*$ must be malnormal
in $D^*/F^*$.  This malnormality obstruction rules out abelian maximal
subgroups in quaternion division algebras over real-closed fields, and hence
in $\mathbb H^*$.

The paper is organized as follows.  Section~\ref{sec:normal-obstruction}
contains the elementary valuative obstruction and the valuation-theoretic
criterion for $1+M_D$ to be nonabelian.  Section~\ref{sec:tame} proves
Theorem~\ref{thm:intro-main} and gives tame examples, including the valued
cyclic and valued quaternion examples of Hazrat--Wadsworth and the classical
examples of Brauer and Dickson.  Section~\ref{sec:beyond-tame} applies the
same method to standard infinite-dimensional valued division rings.  The final
section discusses Hamilton's quaternion division ring and develops the
malnormality obstruction.

\section{The property-$(\mathcal P)$ criterion}
\label{sec:normal-obstruction}

We begin with the elementary group-theoretic observation used throughout the
paper.

\begin{lemma}\label{lem:normal-obstruction}
Let $G$ be a group, and $N$ a normal subgroup of $G$. Suppose that $N$ is nonabelian and $G/N$ is nonabelian. Then $G$ contains no abelian maximal subgroup.
\end{lemma}
\begin{proof}
Suppose, to the contrary, that $A$ is an abelian maximal subgroup of $G$. Since $N$ is normal in $G$, the product $AN$ is a subgroup of $G$. Since $A\leq AN\leq G$ and $A$ is maximal, either $AN=A$ or $AN=G$. If $AN=A$, then $N\subseteq A$, so $N$ is abelian, a contradiction. Thus $AN=G$. It follows that 
$$
G/N=AN/N\cong A/(A\cap N),
$$
which is abelian, again a contradiction. Thus $G$ contains no abelian maximal subgroups.
\end{proof}
For later use, we say that a group $G$ has \textit{property $(\mathcal P)$} if it has a normal subgroup $N$ for which both $N$ and $G/N$ are nonabelian.

We use the standard notation of valuation theory for division rings; see Tignol--Wadsworth~\cite[§1.1.1]{TW}. Let $D$ be a division ring. A valuation on $D$ is a function
$$
        v:D\longrightarrow \Gamma\cup\{\infty\},
$$
where $\Gamma$ is a totally ordered additive abelian group and $\infty$ is a symbol satisfying $\gamma<\infty$ and $\gamma+\infty=\infty+\infty=\infty$ for all $\gamma\in\Gamma$, such that, for all $x,y\in D$,
\begin{enumerate}
        \item $v(x)=\infty$ if and only if $x=0$;
        \item $v(x+y)\geq \min\{v(x),v(y)\}$;
        \item $v(xy)=v(x)+v(y)$.
\end{enumerate}
The value group of $v$ is $\Gamma_D=v(D^*)\subseteq\Gamma$. Thus the restriction $v:D^*\to \Gamma_D$ is a group homomorphism.
Set
$$
        V_D=\{a\in D\mid v(a)\geq 0\}
        \quad
        \text{and}
        \quad
        M_D=\{a\in D\mid v(a)>0\}.
$$
Then $V_D$ is the valuation ring of $D$, and $M_D$ is a two-sided ideal of $V_D$.  Moreover, $M_D$ is the unique maximal ideal of $V_D$, and the quotient
$$
        \overline D=V_D/M_D
$$
is called the residue division ring of $D$.  For $a\in V_D$, we write $\overline a=a+M_D$ for its image in $\overline D$.

Let $\mathbf{gr}(D)$ be the associated graded ring of $D$ determined by $v$. For $\gamma\in\Gamma_D$, set 
$$
D^{\geq\gamma}=\{d\in D^*\mid v(d)\geq\gamma\}\cup\{0\} \quad
\text{and}
\quad
D^{>\gamma}=\{d\in D^*\mid v(d)>\gamma\}\cup\{0\}. 
$$
Then 
$$
\mathbf{gr}(D)=\bigoplus_{\gamma\in\Gamma_D}\mathbf{gr}(D)_\gamma, 
\quad 
\text{where}
\quad
\mathbf{gr}(D)_\gamma=D^{\geq\gamma}/D^{>\gamma}. 
$$
For $a\in D^*$, write $\widetilde a=a+D^{>v(a)}\in \mathbf{gr}(D)_{v(a)}$ and multiplication on homogeneous elements is defined by $\widetilde a\,\widetilde b=\widetilde{ab}$. This multiplication is well-defined and $\mathbf{gr}(D)$ is a graded division ring. Its nonzero homogeneous elements are precisely the units of $\mathbf{gr}(D)$. Thus the map 
\[ 
\rho:D^*\longrightarrow \mathbf{gr}(D)^*
\quad\text{given by}\quad
a\longmapsto \widetilde a, \] is a group epimorphism with kernel $1+M_D$.

We shall use the Fundamental Inequality for valued division algebras given in \cite[Prop.~1.3]{TW}.  Let $D$ be finite-dimensional over its center $K$, and let the valuation on $D$ restrict to the valuation on $K$. Then
\begin{equation}\label{eq:ineqality}
    [D:K] \geq [\overline D:\overline K]|\Gamma_D:\Gamma_K|.
\end{equation}

In particular, we have $|\Gamma_D:\Gamma_K|<\infty$.  The valuation is said to be defectless over $K$ if equality holds; equivalently,
\begin{equation}\label{eq:eqality}
    [D:K]=[\overline D:\overline K]\,|\Gamma_D:\Gamma_K|.
\end{equation}

Thus every tame division algebra over a Henselian valued field satisfies this equality; see \cite[Def.~8.4]{TW}.  

\begin{lemma}\label{lem:finite-index-cofinal}
Let $\Gamma$ be a totally ordered abelian group, and let $\Delta\subseteq\Gamma$ be a subgroup of finite index. If $\Delta$ contains a positive element, then, for each $\gamma\in\Gamma$, there exists $\delta\in\Delta$ such that $\gamma+\delta>0$.
\end{lemma}

\begin{proof}
Choose $0<\varepsilon\in\Delta$. Let $\gamma\in\Gamma$. If $\gamma>0$, then we take $\delta=0$. If $\gamma=0$, then we take $\delta=\varepsilon$. It remains to consider the case $\gamma<0$. Since $\Gamma/\Delta$ is finite, there exists an integer $n\geq 1$ such that $n\gamma\in\Delta$. If $n=1$, then $\gamma\in\Delta$, and we take $\delta=-\gamma+\varepsilon\in\Delta$ so that $\gamma+\delta=\varepsilon>0$. If $n>1$, take $\delta=-n\gamma\in\Delta$. It follows that
$$        \gamma+\delta
        =
        \gamma-n\gamma
        =
        -(n-1)\gamma.$$

Since $\gamma<0$, we have $-(n-1)\gamma>0$. Hence $\gamma+\delta>0$. The proof is now complete.
\end{proof}

\begin{theorem}
\label{thm:normals-nonabelian}
Let $D$ be a division ring finite-dimensional over its center $K$. Suppose that $D$ is endowed with a valuation extending a nontrivial valuation on $K$. Then $1+M_D$ is nonabelian if and only if $D$ is noncommutative.
\end{theorem}

\begin{proof}
If $1+M_D$ is nonabelian, then $D^*$ is nonabelian, and so $D$ is noncommutative. Conversely, suppose that $D$ is noncommutative. Choose $x,y\in D^*$ such that $xy\neq yx$. Since $D$ is finite-dimensional over $K$, the inequality (\ref{eq:ineqality}) shows that the group $\Gamma_D/\Gamma_K$ is finite. Since the valuation on $K$ is nontrivial, there exists $a\in K^*$ such that $v(a)\neq 0$. If $v(a)>0$, then $v(a)$ is a positive element of
$\Gamma_K$. If $v(a)<0$, then $v(a^{-1})=-v(a)>0$. Hence $\Gamma_K$ contains a positive element. By Lemma~\ref{lem:finite-index-cofinal}, there exist $c,d\in K^*$ such that $v(cx)>0$ and $v(dy)>0$. This implies that $cx,dy\in M_D$.

Since $c,d\in K^*$ are central, we have
$$        
(cx)(dy)-(dy)(cx)=cd(xy-yx)\neq 0.
$$
Thus $M_D$ contains two noncommuting elements, from which it follows that $1+M_D$ is nonabelian.
\end{proof}

\begin{corollary}
\label{cor:normal-obstruction}
Let $D$ be a division ring finite-dimensional over its center $K$, and suppose
that $D$ carries a valuation extending a nontrivial valuation on $K$.  If $D$
is noncommutative and $\mathbf{gr}(D)^*$ is nonabelian, then $D^*$ contains
no abelian maximal subgroup.
\end{corollary}

\begin{proof}
According to Theorem~\ref{thm:normals-nonabelian} we conclude that that $1+M_D$ is nonabelian. Moreover, as
$$
        D^*/(1+M_D)\cong \mathbf{gr}(D)^*, 
$$
which is non-abelian, we conclude that $D^*$ has property $(\mathcal P)$. Thus, the result follows from Lemma \ref{lem:normal-obstruction}.
\end{proof}

\section{Tame division algebras}
\label{sec:tame}
Let $(K,v)$ be a Henselian valued field with nontrivial valuation, and let
$D$ be a finite-dimensional central division algebra over $K$. The valuation
$v$ extends uniquely to a valuation on $D$. We keep the notation $V_D$,
$M_D$, $\overline D$, and $\mathbf{gr}(D)$ introduced above, and we write
$\mathbf{gr}(K)$ for the associated graded field of $K$. The division algebra $D$
is tame over $K$ if and only if
\begin{equation}\label{eq:tame}
    [\mathbf{gr}(D):\mathbf{gr}(K)]=[D:K]
    \quad\text{and}\quad
    Z(\mathbf{gr}(D))=\mathbf{gr}(K).
\end{equation}

Equivalently, $D$ is defectless over $K$ and the center of the associated
graded division algebra is as small as possible. This is the form of the
tameness condition needed below.

\begin{lemma}
\label{lem:tame-graded-noncommutative}
Let $K$ be a Henselian valued field with nontrivial valuation, and let $D$ be a tame finite-dimensional central division algebra over $K$. If $D$ is
noncommutative, then $\mathbf{gr}(D)$ is noncommutative. Consequently,
$\mathbf{gr}(D)^*$ is nonabelian.
\end{lemma}

\begin{proof}
Since $D$ is noncommutative, we get $[D:K]>1$; thus, by equation (\ref{eq:tame}) we get $[\mathbf{gr}(D):\mathbf{gr}(K)]>1$. It follows from that $\mathbf{gr}(D)$ is noncommutative. Finally, if $\mathbf{gr}(D)^*$ were abelian, then all nonzero homogeneous
elements of $\mathbf{gr}(D)$ would commute, which would make $\mathbf{gr}(D)$
commutative, a contradiction. Hence $\mathbf{gr}(D)^*$ is nonabelian.
\end{proof}

\begin{theorem}\label{thm:main-tame}
Let $K$ be a Henselian valued field with nontrivial valuation, and let $D$be a tame finite-dimensional central division algebra over $K$. If $D$ is noncommutative, then $D^*$ contains no abelian maximal subgroup.
\end{theorem}
\begin{proof}
Since $D$ is noncommutative and finite-dimensional over its center $K$, and since the valuation on $K$ is nontrivial, Theorem~\ref{thm:normals-nonabelian} shows that $1+M_D$ is nonabelian. By Lemma~\ref{lem:tame-graded-noncommutative}, $\mathbf{gr}(D)^*$ is nonabelian. It follows that $D^*/(1+M_D)\cong\mathbf{gr}(D)^*$ is nonabelian. Therefore $D^*$ has property $(\mathcal P)$, and so $D^*$ contains no abelian maximal subgroup.
\end{proof}

\begin{corollary}
\label{cor:no-locally-nilpotent-maximal}
Let $K$ be a Henselian valued field with nontrivial valuation, and let $D$ be a noncommutative tame finite-dimensional central division algebra over $K$. Then $D^*$ contains no locally nilpotent maximal subgroup.
\end{corollary}

\begin{proof}
By Theorem~\ref{thm:main-tame}, $D^*$ contains no abelian maximal subgroup.
On the other hand, it was proved in \cite{KH2} that every locally nilpotent maximal subgroup of the multiplicative group of a division ring is abelian. Hence $D^*$ contains no locally nilpotent maximal subgroup.
\end{proof}

We next recall a standard family of tame cyclic division algebras given in \cite[Example 8]{HaW1}, which nicely illustrates Theorem~\ref{thm:main-tame}.

\begin{example}
\label{ex:valued-cyclic}
Let $K$ be a Henselian field with a discrete rank-one valuation, 
so that $\Gamma_K=\mathbb Z$.  Let $L/K$ be an unramified cyclic Galois
extension of degree $n>1$, with $\operatorname{Gal}(L/K)=\langle\sigma\rangle$.
Let $\pi\in K^*$ with $v(\pi)=1$, and let $D=(L/K,\sigma,\pi)$ be the cyclic algebra associated with $(L/K,\sigma)$ and $\pi$, so that $D=\bigoplus_{i=0}^{n-1}Lx^i$, in which
\begin{equation}\label{eq:cyclic-multiplication}
    xcx^{-1}=\sigma(c)\ \text{for all }c\in L\quad\text{and}\quad x^n=\pi.
\end{equation}
It was shown in \cite[Example~8]{HaW1} that the
valuation extends to $D$ by setting
\begin{equation}\label{eq:valuation fomula}
     v\left(\sum_{i=0}^{n-1}c_i x^i\right)
        =
        \min_{0\leq i\leq n-1}\left\{v(c_i)+\frac{i}{n}\right\}.
\end{equation}
With this valuation, we have
\[
        \overline D=\overline L
        \quad\text{and}\quad
        \Gamma_D=\frac1n\mathbb Z.
\]
Hence
\[
        [\overline D:\overline K]\,|\Gamma_D:\Gamma_K|
        =
        n^2
        =
        [D:K].
\]
Moreover, $Z(\overline D)=\overline L$ is separable over $\overline K$, and
so $D$ is a tame
semiramified division algebra.  Since $n>1$, the algebra $D$ is
noncommutative, and Theorem~\ref{thm:main-tame} gives that $D^*$ contains no
abelian maximal subgroup. It was also shown that
\[
        D^*/K^*(1+M_D)
        \cong
        \overline L^*/\overline K^*\rtimes \mathbb Z/n\mathbb Z,
\]
where the generator of $\mathbb Z/n\mathbb Z$ acts through the residue
automorphism induced by $\sigma$.  If $K$ is a local field with residue field
$\mathbb F_q$, this becomes
\[
        \left(\mathbb Z/\frac{q^n-1}{q-1}\mathbb Z\right)
        \rtimes \mathbb Z/n\mathbb Z,
\]
with the generator acting by multiplication by $q$.  For $n=2$, this is the
dihedral group of order $2(q+1)$.

In what follows, we also compute the group $1+M_K$ and $1+M_D$.  Since $L/K$ is
unramified, the valuation ring of $L$ is $V_L$, its maximal ideal is
$M_L=\pi V_L$, and $\Gamma_L=\Gamma_K=\mathbb Z$.  Hence
\[
        M_K=\pi V_K
        \quad
        \text{and}
        \quad
        1+M_K=1+\pi V_K.
\]
The formula (\ref{eq:valuation fomula}) implies that
\[
        V_D=V_L\oplus V_Lx\oplus\cdots\oplus V_Lx^{n-1}
\]
and
\[
        M_D=\pi V_L\oplus V_Lx\oplus\cdots\oplus V_Lx^{n-1}.
\]
Equivalently, since $x^n=\pi$, one has
\[
        M_D=xV_D=V_Dx.
\]
Thus
\[
        1+M_D
        =
        \left\{
        1+a_0+a_1x+\cdots+a_{n-1}x^{n-1}
        \mid
        a_0\in \pi V_L,\ a_i\in V_L\text{ for }1\leq i\leq n-1
        \right\}.
\]
This is equivalent to,
\[
        1+M_D
        =
        \left\{
        b_0+b_1x+\cdots+b_{n-1}x^{n-1}
        \mid
        b_0\in 1+\pi V_L,\ b_i\in V_L\text{ for }1\leq i\leq n-1
        \right\}.
\]

The multiplication in $1+M_D$ is determined by the relations (\ref{eq:cyclic-multiplication}): If $u=\sum_{i=0}^{n-1}b_i x^i$ and $w=\sum_{j=0}^{n-1}d_j x^j$, with $b_0,d_0\in 1+\pi V_L$ and $b_i,d_j\in V_L$ for $i,j\geq 1$, then
$uw=\sum_{r=0}^{n-1}e_r x^r$, where
\[
        e_r=
        \sum_{\substack{0\leq i,j\leq n-1\\ i+j\equiv r\pmod n}}
        b_i\sigma^i(d_j)\pi^{\lfloor (i+j)/n\rfloor}.
\]
In particular, $e_0\in 1+\pi V_L$ and $e_r\in V_L$ for $r\geq 1$, so
$uw\in 1+M_D$.

We can also prove that $1+M_D$ is nonabelian.  Choose
$c\in V_L$ such that $\overline{\sigma(c)}\ne \overline c$ in $\overline L$.
Then $1+x$ and $1+cx$ lie in $1+M_D$, and
\[
        (1+x)(1+cx)-(1+cx)(1+x)=(\sigma(c)-c)x^2\ne 0.
\]
In the special case $K=\mathbb Q_p$, we take $\pi=p$, $V_K=\mathbb Z_p$,
and $M_K=p\mathbb Z_p$.  Thus
\[
        1+M_K=1+p\mathbb Z_p.
\]
Let $L/\mathbb Q_p$ be the unramified extension of degree $n$, and let
$\mathcal O_L=V_L$ be its ring of integers.  Then $M_L=p\mathcal O_L$ and
$\overline L\cong\mathbb F_{p^n}$.  For $D=(L/\mathbb Q_p,\sigma,p)$, we have
\[
        V_D=\mathcal O_L\oplus \mathcal O_Lx\oplus\cdots\oplus
        \mathcal O_Lx^{n-1},
\]
and
\[
        M_D=p\mathcal O_L\oplus \mathcal O_Lx\oplus\cdots\oplus
        \mathcal O_Lx^{n-1}=xV_D=V_Dx.
\]
Consequently,
\[
        1+M_D
        =
        \left\{
        b_0+b_1x+\cdots+b_{n-1}x^{n-1}
        \mid
        b_0\in 1+p\mathcal O_L,\ b_i\in\mathcal O_L\text{ for }1\leq i\leq n-1
        \right\}.
\]
It follows that
\[
        V_K^*/(1+M_K)\cong\mathbb F_p^*
        \quad
        \text{and}
        \quad
        V_D^*/(1+M_D)\cong\mathbb F_{p^n}^*.
\]
Moreover, if $p$ is odd, then we have
\[
        1+p\mathbb Z_p\cong p\mathbb Z_p\cong\mathbb Z_p.
\]
For $p=2$, one has
\[
        1+2\mathbb Z_2=\{\pm 1\}\times (1+4\mathbb Z_2)
        \quad
        \text{and}
        \quad
        1+4\mathbb Z_2\cong\mathbb Z_2.
\]
Finally, one has
\[
        1+M_D\supseteq 1+M_D^2\supseteq 1+M_D^3\supseteq\cdots
        \quad\text{and}\quad M_D^r=x^rV_D,
\]
and, for every $r\geq 1$,
\[
        (1+M_D^r)/(1+M_D^{r+1})
        \cong
        M_D^r/M_D^{r+1}
        \cong
        \mathbb F_{p^n}.
\]
\end{example}

\begin{example}[A valued quaternion division algebra]
\label{ex:valued-quaternion}
Let $K_0=\mathbb R((x))$, ordered so that $x>0$, and let $\mathcal R$ be a real closure of $K_0$ with respect to this ordering. Define a tower of subfields of $\mathcal R$ by 
\[ 
K_0\subseteq K_1\subseteq K_2\subseteq\cdots 
\quad 
\text{and}
\quad
K_{r+1}=K_r\bigl(\{\sqrt c\mid c\in K_r,\ c>0\}\bigr). 
\] 
The Euclidean hull of $K_0$ in $\mathcal R$ is \[ K=\bigcup_{r\geq 0}K_r. \] The ordering of $\mathcal R$ restricts to an ordering of $K$, and every positive element of $K$ is a square. Thus $K$ is Euclidean. It was shown in \cite[Example~7]{HaW1} that the $x$-adic valuation on $K_0$ extends to a Henselian valuation on $K$ with 
\[ 
\Gamma_K\cong \mathbb Z[1/2]
\quad
\text{and}
\quad
\overline K\cong \mathbb R. 
\] 
Let 
\[ 
D=\left(\frac{-1,-1}{K}\right) = K\oplus Ki\oplus Kj\oplus Kij, 
\]
where $i^2=j^2=-1$ and $ij=-ji$. The valuation on $K$ extends uniquely to $D$, so that
\[
        \overline D\cong
        \left(\frac{-1,-1}{\mathbb R}\right)
        =
        \mathbb H
        \quad\text{and}\quad
        \Gamma_D=\Gamma_K.
\]
Thus $D$ is inertial over $K$.  In particular,
\[
        [\overline D:\overline K]\,|\Gamma_D:\Gamma_K|
        =
        4
        =
        [D:K].
\]
Since $\operatorname{char}(\overline K)=0$, the
division algebra $D$ is tame.

The valuation can be described explicitly on the standard quaternion basis.
For
\[
        a=a_0+a_1i+a_2j+a_3ij\in D,
\]
where $a_r\in K$, one has
\[
        v_D(a)=\min\{v_K(a_0),v_K(a_1),v_K(a_2),v_K(a_3)\}.
\]
Consequently
\[
        V_D=V_K\oplus V_Ki\oplus V_Kj\oplus V_Kij
\]
and
\[
        M_D=M_K\oplus M_Ki\oplus M_Kj\oplus M_Kij.
\]
Here
\[
        M_K=\{a\in K\mid v_K(a)>0\},
        \quad\text{and}\quad
        1+M_K=\{a\in V_K^*\mid \overline a=1\}.
\]
Since $\Gamma_K\cong\mathbb Z[1/2]$ is not discrete, $M_K$ is not generated by a single element.  The subgroup $1+M_D$ of $D^*$ is given by
\[
        1+M_D
        =
        \left\{
        b_0+b_1i+b_2j+b_3ij
        \mid
        b_0\in 1+M_K,\ b_1,b_2,b_3\in M_K
        \right\}.
\]
Equivalently,
\[
        1+M_D
        =
        1+(M_K\oplus M_Ki\oplus M_Kj\oplus M_Kij).
\]

This example lies in the tame case.  Since $D$ is noncommutative, Theorem~\ref{thm:main-tame} gives that $D^*$ contains no abelian maximal
subgroup.  The noncommutativity of $1+M_D$ is also proved directly: if $0\ne t\in M_K$, then $1+ti$ and $1+tj$ lie in $1+M_D$, and
\[
        (1+ti)(1+tj)-(1+tj)(1+ti)=2t^2ij\ne 0.
\]
\end{example}

We next give two classical cyclic division algebras, due to Brauer and Dickson, in a form suited to the valuative obstruction.  The original algebras are not defined over Henselian centers.  In each case, however, the center carries a natural nontrivial valuation that extends to the algebra. After passing to a Henselization or completion, Morandi's theorem preserves the associated graded division algebra; see \cite[Th.~4.1]{TW}. Thus the resulting Henselian examples fall under Theorem~\ref{thm:main-tame}, while the original algebras are covered by Corollary~\ref{cor:normal-obstruction}.

\begin{example}[Brauer's cyclic division algebra]
\label{ex:brauer-cyclic}
Let $E=\mathbb Q(y,z)$, where $y$ and $z$ are independent commuting indeterminates over $\mathbb Q$.  Let $\sigma$ be the $\mathbb Q$-automorphism of $E$ given by $\sigma(y)=z$ and $\sigma(z)=-y$.  Then $\sigma$ has order
$4$.  Put $F=E^\sigma$, and consider
\[
        D_0=(E/F,\sigma,-1)
        =
        E\oplus Ex\oplus Ex^2\oplus Ex^3,
\]
where $x^4=-1$ and $xc=\sigma(c)x$ for $c\in E$.  Lam presents in \cite[pp.~224--226]{L} Brauer's example and proves that $D_0$ is a division algebra of degree $4$ over $F$.

We define a valuation on $D_0$ from the total degree in $y$ and $z$.  Let
$P=\mathbb Q[y,z]$, and set 
$$
R=P\oplus Px\oplus Px^2\oplus Px^3\subseteq D_0.
$$
Give $R$ the grading by total degree in $y,z$, with $x$ of degree $0$.  Thus $R_d$ is the $\mathbb Q$-span of the monomials $y^m z^n x^i$ with $m+n=d$ and $0\leq i\leq 3$.  It was shown in \cite{L} that $R$ is a graded domain.  Hence, if $\delta(r)$ denotes the largest total degree occurring in a nonzero element $r\in R$, then $\delta(rs)=\delta(r)+\delta(s)$ for all nonzero $r,s\in R$.

Every element of $D_0$ has the form $q^{-1}r$ with $0\ne q\in P$ and
$r\in R$.  Define
\[
        v(q^{-1}r)=\delta(q)-\delta(r).
\]
This is well-defined: if $q^{-1}r=q'^{-1}r'$, then $q'r=qr'$, and the multiplicativity of $\delta$ gives
$\delta(q')+\delta(r)=\delta(q)+\delta(r')$.  The same multiplicativity shows that $v$ is a valuation on $D_0$.  Its restriction to $F$ is nontrivial, because $y^2+z^2\in F$ and $v(y^2+z^2)=-2$.

Let $F^h$ be a Henselization of $F$ with respect to $v|_F$, and set $D=D_0\otimes_F F^h$.  By Morandi's theorem, $D$ is a division algebra and $\mathbf{gr}(D)\cong\mathbf{gr}(D_0)$.  The residue field has characteristic $0$, so $D$ is tame over $F^h$ by  \cite[Def.~8.4]{TW}.

Since $xy=zx$, we have $\widetilde x\,\widetilde y=\widetilde z\,\widetilde x$ in $\mathbf{gr}(D)$.  But $\widetilde y\ne\widetilde z$, since $v(y)=v(z)=-1$ and $v(y-z)=-1$.  As $\widetilde x \in \mathbf{gr}(D)^*$, it follows that $\widetilde x\,\widetilde y\ne\widetilde y\,\widetilde x$.  Hence $\mathbf{gr}(D)$ is noncommutative.  By Theorem~\ref{thm:main-tame}, $D^*$ contains no abelian maximal subgroup.

The same graded computation, we get that $\mathbf{gr}(D_0)^*$ is nonabelian, and Corollary~\ref{cor:normal-obstruction} shows that $D_0^*$ contains no abelian maximal subgroup.
\end{example}

\begin{example}[Dickson's cyclic division algebra]
\label{ex:dickson-cyclic}
Let $f(t)=t^3+t^2-2t-1\in\mathbb Q[t]$, let $u$ be a root of $f$, and put $E=\mathbb Q(u)$.  Then $E/\mathbb Q$ is a cyclic cubic extension.  Let $\sigma$ be the generator determined by $\sigma(u)=u^2-2$.  Dickson's algebra is
\[
        D_0=(E/\mathbb Q,\sigma,2)=E\oplus Ex\oplus Ex^2,
\]
where $x^3=2$ and $xu=(u^2-2)x$.  It was shown in \cite[p.~227]{L} that $D_0$ is a $9$-dimensional division algebra over $\mathbb Q$.

Since $f(t)$ is irreducible modulo $2$, the field $E_2=\mathbb Q_2(u)$ is an unramified extension of $\mathbb Q_2$.  Set
\[
        D=D_0\otimes_{\mathbb Q}\mathbb Q_2
        \cong
        (E_2/\mathbb Q_2,\sigma,2).
\]
Since $E_2/\mathbb Q_2$ is unramified of degree $3$, for every
$\alpha\in E_2^*$ one has
\[
        v_{\mathbb Q_2}\bigl(N_{E_2/\mathbb Q_2}(\alpha)\bigr)
        =
        3v_{E_2}(\alpha)\in 3\mathbb Z.
\]
As $v_{\mathbb Q_2}(2)=1$, it follows that $2\notin N_{E_2/\mathbb Q_2}(E_2^*)$, and so it is a division algebra by \cite[Corollary 14.8]{L}.

Let $v_2$ be the $2$-adic valuation on $\mathbb Q_2$. Since $E_2/\mathbb Q_2$ is unramified, $v_2$ has a unique extension to $E_2$; denote it by $v_{E_2}$. Thus
\[
        v_{E_2}|_{\mathbb Q_2}=v_2,
        \quad
        v_{E_2}(2)=1,
        \quad\text{and}\quad
        v_{E_2}(E_2^*)=\mathbb Z.
\]
The valuation on $D$ is
\[
        v\left(c_0+c_1x+c_2x^2\right)
        =
        \min\left\{
        v_{E_2}(c_0),\,
        v_{E_2}(c_1)+\frac13,\,
        v_{E_2}(c_2)+\frac23
        \right\}.
\]
Thus $v(x)=1/3$, $\Gamma_D=(1/3)\mathbb Z$, and $\Gamma_{\mathbb Q_2}=\mathbb Z$.  Also $\overline D=\overline{E_2}\cong\mathbb F_8$.  Therefore
\[
        [\overline D:\mathbb F_2]\,|\Gamma_D:\Gamma_{\mathbb Q_2}|
        =
        9
        =
        [D:\mathbb Q_2].
\]
Since the residue characteristic is $2$ and $2\nmid \deg D=3$, the division algebra $D$ is tame over $\mathbb Q_2$.

The associated graded algebra is noncommutative.  Let $\widetilde u$ be the residue of $u$ in $\overline{E_2}\cong\mathbb F_8$, and let $\widetilde x$ be the image of $x$ in $\mathbf{gr}(D)$.  From $xu=\sigma(u)x=(u^2-2)x$ we get
\[
        \widetilde x\,\widetilde u
        =
        \widetilde u^{\,2}\widetilde x
\]
in $\mathbf{gr}(D)$.  Since $\widetilde u\notin\mathbb F_2$, we have $\widetilde u^{\,2}\ne\widetilde u$.  Hence $\widetilde x\,\widetilde u\ne\widetilde u\,\widetilde x$, and so $\mathbf{gr}(D)$ is noncommutative.  By Theorem~\ref{thm:main-tame}, $D^*$ contains no abelian maximal subgroup.
\end{example}

\section{Beyond finite-dimensional tame algebras}
\label{sec:beyond-tame}

The preceding section treated finite-dimensional tame division algebras over Henselian centers.  The same argument applies to many valued division rings that need not be finite-dimensional over their centers.  To do this, we first record a graded criterion that is often useful for detecting the
failure of multiplicative groups of valued division rings to have abelian
maximal subgroups.

\begin{proposition}
\label{prop:positive-positive-graded}
Let $D$ be a valued division ring, and let $E=\mathbf{gr}(D)$. Suppose that
there exist homogeneous elements $u\in E_\alpha^*$ and $w\in E_\beta^*$,
with $\alpha>0$ and $\beta>0$, such that $uw\ne wu$. Then $D^*$ contains no
abelian maximal subgroup.
\end{proposition}

\begin{proof}
Choose $a,b\in D^*$ with $v(a)=\alpha$, $v(b)=\beta$,
$a+D^{>\alpha}=u$, and $b+D^{>\beta}=w$. Then $a,b\in M_D$, so
$1+a,1+b\in 1+M_D$. If $ab=ba$, then their images in
$E_{\alpha+\beta}$ would be equal, contrary to $uw\ne wu$. Hence
$ab\ne ba$, and therefore $1+a$ and $1+b$ do not commute. Thus $1+M_D$ is
nonabelian.

Also $E^*$ is nonabelian because $u$ and $w$ do not commute. Since
$D^*/(1+M_D)\cong E^*$, both $1+M_D$ and the quotient are nonabelian.
Thus $D^*$ has property $(\mathcal P)$, and the result follows from
Corollary \ref{cor:normal-obstruction}.
\end{proof}

\begin{corollary}
\label{cor:canonical-action}
Let \(D\) be a valued division ring with center \(K\). Suppose that the
canonical action
\[
        \theta_D:\Gamma_D/\Gamma_K
        \longrightarrow
        \Aut_{\overline K}(Z(\overline D))
\]
is nontrivial. Then \(D^*\) contains no abelian maximal subgroup.
\end{corollary}

\begin{proof}
Since the action is nontrivial, there exist
$\gamma+\Gamma_K\in\Gamma_D/\Gamma_K$ and $z\in Z(\overline D)^*$ such that
$\theta_D(\gamma+\Gamma_K)(z)\ne z$. Replacing $\gamma$ by $-\gamma$, if
necessary, we may assume $\gamma>0$. Choose $d\in D^*$ with $v(d)=\gamma$,
and choose $a\in V_D^*$ lifting $z$. Then in $\mathbf{gr}(D)$,
\[
        \widetilde d\,\widetilde a\,\widetilde d^{-1}
        =
        \theta_D(\gamma+\Gamma_K)(z)
        \ne z.
\]
Thus $\widetilde d\,\widetilde a\ne \widetilde a\,\widetilde d$. Put
$u=\widetilde d$ and $w=\widetilde a\,\widetilde d$. Then
$u,w\in \mathbf{gr}(D)_\gamma^*$, and $uw\ne wu$. The result follows from
Proposition~\ref{prop:positive-positive-graded}.
\end{proof}

We now apply Proposition~\ref{prop:positive-positive-graded} to three
standard families.   For the relevant constructions, see Tignol--Wadsworth~\cite[Sections~1.1.3--1.1.4]{TW} and Lam~\cite[Section~14]{L}. The first one is the twisted Laurent series given in \cite[§1.1.2, pp.~3--4]{TW}:

\begin{proposition}
\label{prop:twisted-laurent}
Let $A$ be a division ring, let $\sigma\in\operatorname{Aut}(A)$, and let
$D=A((x;\sigma))$ be the twisted Laurent series, with multiplication determined by $xa=\sigma(a)x$ for
$a\in A$. Suppose that either $A^*$ is nonabelian or
$\sigma\ne\operatorname{id}_A$. Then $D^*$ contains no abelian maximal
subgroup.
\end{proposition}

\begin{proof}
We use the $x$-adic valuation on $D$. Thus
$$
        V_D=A[[x;\sigma]],\quad
        M_D=xA[[x;\sigma]],\quad
        \overline D=A,\quad\text{and}\quad
        \Gamma_D=\mathbb Z.
$$
Moreover,
\[
        \mathbf{gr}(D)\cong A[\widetilde x,\widetilde x^{-1};\sigma],
\]
where $\widetilde x a=\sigma(a)\widetilde x$ for $a\in A$.
Assume first that $\sigma\ne\operatorname{id}_A$. Choose $a\in A$ with
$\sigma(a)\ne a$. Set $u=\widetilde x$ and $w=a\widetilde x$. Then
$u,w\in \mathbf{gr}(D)_1^*$, and
\[
        uw=\sigma(a)\widetilde x^2
        \quad\text{and}\quad
        wu=a\widetilde x^2.
\]
Hence $uw\ne wu$. Proposition~\ref{prop:positive-positive-graded} applies.
It remains to consider the case $\sigma=\operatorname{id}_A$ and $A^*$ is
nonabelian. Choose $a,b\in A^*$ with $ab\ne ba$. Set
$u=a\widetilde x$ and $w=b\widetilde x$. Then
$u,w\in \mathbf{gr}(D)_1^*$, and
\[
        uw=ab\widetilde x^2
        \quad\text{and}\quad
        wu=ba\widetilde x^2.
\]
Thus $uw\ne wu$. Again Proposition~\ref{prop:positive-positive-graded}
applies. Therefore $D^*$ contains no abelian maximal subgroup.
\end{proof}

For iterated Laurent series and their natural valuations, we follow
Tignol--Wadsworth~\cite[§1.1.3, pp.~5--7]{TW}.

\begin{proposition}
\label{prop:iterated-laurent}
Let $M$ be a field, let $\sigma=(\sigma_i)^n_{i=1}$ be a family of pairwise
commuting automorphisms of $M$, and let $u=(u_{i,j})^n_{i,j=1}$ be a family
of elements of $M^*$ satisfying
$$
        u_{i,i}=1,\quad u_{i,j}u_{j,i}=1,\quad\text{and}\quad
        u_{i,j}u_{j,k}u_{k,i}
        =
        \sigma_k(u_{i,j})\sigma_i(u_{j,k})\sigma_j(u_{k,i})
        \quad\text{for all } i,j,k.
$$
Let
\[
        D=L((M;\sigma,u))
        =
        M((x_1;\sigma_1))((x_2;\widehat\sigma_2))\cdots
        ((x_n;\widehat\sigma_n))
\]
be the iterated Laurent series division ring.  Thus
$$
        x_i m=\sigma_i(m)x_i\quad\text{and}\quad x_i x_j=u_{i,j}x_jx_i
        \quad\text{for } m\in M \text{ and } 1\leq i,j\leq n.
$$  Suppose that at least one of the
following conditions holds:
\begin{enumerate}[label=\textup{(\arabic*)}]
        \item $\sigma_i\ne\operatorname{id}_M$ for some $i$;
        \item $u_{i,j}\ne 1$ for some $i\ne j$.
\end{enumerate}
Then $D^*$ contains no abelian maximal subgroup.
\end{proposition}

\begin{proof}
Let $v=v_{x_1,\ldots,x_n}$ be the $(x_1,\ldots,x_n)$-adic valuation on $D$.
Its value group is $\mathbb Z^n$, ordered right-to-left, and
$v(x_i)=\epsilon_i>0$, where $\epsilon_i$ is the $i$-th standard basis vector.
For each $i$, put
$\widetilde x_i=x_i+D^{>v(x_i)}\in \mathbf{gr}(D)_{v(x_i)}$.
Assume first that $\sigma_i\ne\operatorname{id}_M$ for some $i$. Choose
$a\in M$ with $\sigma_i(a)\ne a$. Set $u=\widetilde x_i$ and
$w=a\widetilde x_i$. Then $u,w\in \mathbf{gr}(D)_{v(x_i)}^*$, and
\[
        uw=\sigma_i(a)\widetilde x_i^2
        \quad\text{and}\quad
        wu=a\widetilde x_i^2.
\]
Hence $uw\ne wu$. Proposition~\ref{prop:positive-positive-graded} applies.
Assume next that $u_{i,j}\ne 1$ for some $i\ne j$. Set
$u=\widetilde x_i$ and $w=\widetilde x_j$. Then
$u\in \mathbf{gr}(D)_{v(x_i)}^*$ and
$w\in \mathbf{gr}(D)_{v(x_j)}^*$, both degrees are positive, and $uw=u_{i,j}wu$. Since $u_{i,j}\ne 1$, we have $uw\ne wu$. Again
Proposition~\ref{prop:positive-positive-graded} applies. Therefore $D^*$
contains no abelian maximal subgroup.
\end{proof}

For Mal'cev--Neumann series and their natural valuations, we follow
Tignol--Wadsworth~\cite[§1.1.4, pp.~7--9]{TW}.

\begin{proposition}
\label{prop:malcev-neumann}
Let $A$ be a division ring, and let $\Gamma$ be a totally ordered additive
abelian group.  Let $\omega:\Gamma\to \operatorname{Aut}(A)$ and
$f:\Gamma\times\Gamma\to A^*$ satisfy
\[
        \omega_\gamma(f(\delta,\varepsilon))f(\gamma,\delta+\varepsilon)
        =
        f(\gamma,\delta)f(\gamma+\delta,\varepsilon)
        \quad\text{and}\quad
        \omega_\gamma\omega_\delta(a)
        =
        f(\gamma,\delta)\omega_{\gamma+\delta}(a)f(\gamma,\delta)^{-1},
\]
for all $\gamma,\delta,\varepsilon\in\Gamma$ and $a\in A$, and
$\omega_0=\operatorname{id}_A$, $f(0,\gamma)=f(\gamma,0)=1$.  Let
\[
        D=A((\Gamma;\omega,f))
\]
be the corresponding Mal'cev--Neumann division ring.  Thus an element of
$D$ is a formal series $\varphi=\sum_{\gamma\in\Gamma}a_\gamma z^\gamma$ with well-ordered support, and multiplication is determined by
\[
        z^\gamma a=\omega_\gamma(a)z^\gamma\quad\text{and}\quad
        z^\gamma z^\delta=f(\gamma,\delta)z^{\gamma+\delta}.
\]
Suppose that at least one of the following conditions holds:
\begin{enumerate}[label=\textup{(\arabic*)}]
        \item $\omega_\gamma(a)\ne a$ for some $\gamma>0$ and $a\in A$;
        \item $f(\gamma,\delta)\ne f(\delta,\gamma)$ for some
        $\gamma,\delta>0$;
        \item $a\omega_\gamma(b)\ne b\omega_\gamma(a)$ for some
        $\gamma>0$ and $a,b\in A$.
\end{enumerate}
Then $D^*$ contains no abelian maximal subgroup.
\end{proposition}

\begin{proof}
Let $v$ be the Mal'cev--Neumann valuation on $D$:
\[
        v(\varphi)=\min(\operatorname{supp}\varphi)\quad\text{and}\quad v(0)=\infty.
\]
Then $\Gamma_D=\Gamma$, $\overline D=A$, and $z^\gamma\in M_D$ whenever
$\gamma>0$. For $\gamma>0$, write
$\widetilde z^\gamma=z^\gamma+D^{>\gamma}\in \mathbf{gr}(D)_\gamma$.

Assume first that $\omega_\gamma(a)\ne a$ for some $\gamma>0$ and $a\in A$.
Set $u=\widetilde z^\gamma$ and $w=a\widetilde z^\gamma$. Then
$u,w\in \mathbf{gr}(D)_\gamma^*$, and
\[
        uw=\omega_\gamma(a)f(\gamma,\gamma)\widetilde z^{2\gamma}
        \quad\text{and}\quad
        wu=af(\gamma,\gamma)\widetilde z^{2\gamma}.
\]
Hence $uw\ne wu$. Proposition~\ref{prop:positive-positive-graded} applies.
Assume next that $f(\gamma,\delta)\ne f(\delta,\gamma)$ for some
$\gamma,\delta>0$. Set $u=\widetilde z^\gamma$ and
$w=\widetilde z^\delta$. Then $u\in \mathbf{gr}(D)_\gamma^*$ and
$w\in \mathbf{gr}(D)_\delta^*$, and
\[
        uw=f(\gamma,\delta)\widetilde z^{\gamma+\delta}
        \quad\text{and}\quad
        wu=f(\delta,\gamma)\widetilde z^{\gamma+\delta}.
\]
Thus $uw\ne wu$. Again Proposition~\ref{prop:positive-positive-graded}
applies.
Finally, assume that $a\omega_\gamma(b)\ne b\omega_\gamma(a)$ for some
$\gamma>0$ and $a,b\in A$. Set $u=a\widetilde z^\gamma$ and
$w=b\widetilde z^\gamma$. Then $u,w\in \mathbf{gr}(D)_\gamma^*$, and
\[
        uw=a\omega_\gamma(b)f(\gamma,\gamma)\widetilde z^{2\gamma}
        \quad\text{and}\quad
        wu=b\omega_\gamma(a)f(\gamma,\gamma)\widetilde z^{2\gamma}.
\]
Hence $uw\ne wu$. Proposition~\ref{prop:positive-positive-graded} applies.
Therefore $D^*$ contains no abelian maximal subgroup.
\end{proof}

\section{A contrasting example and the malnormality obstruction}
\label{sec:hamilton}
\label{sec:malnormal}

The valuative obstruction above works by producing a nonabelian normal
subgroup and a nonabelian quotient.  This mechanism is not automatic for
multiplicative groups of noncommutative division rings.  Hamilton's quaternion
division ring is the simplest contrast.  It is a cyclic division algebra of
degree $2$ over $\mathbb R$, but its multiplicative group does not have
property $(\mathcal P)$.  Nevertheless, $\mathbb H^*$ has no abelian maximal
subgroup.  This leads to the malnormality obstruction, based on the position of a
maximal subfield inside $D^*$.

\begin{example}
\label{ex:Hamilton-no-P}
Let $\mathbb H=\left(\frac{-1,-1}{\mathbb R}\right)$ be the division ring of real quaternions. Then $\mathbb H^*$
does not have property $(\mathcal P)$.
\end{example}

\begin{proof}
Every nonzero quaternion $q=a+bi+cj+dk\in\mathbb H^*$, with
$a,b,c,d\in\mathbb R$, has a unique polar decomposition $q=ru$, where
$$
        r=|q|=\sqrt{q\overline q}
        =\sqrt{a^2+b^2+c^2+d^2}\in\mathbb R_{>0}
$$
and $u=q/|q|$. Here $\overline q=a-bi-cj-dk$ is the standard quaternion
conjugate of $q$, and $\mathbb R_{>0}$ denotes the multiplicative group of
positive real numbers. Thus $u$ lies in
$$
        S^3=\{v\in\mathbb H\mid |v|=1\},
$$
the group of unit quaternions. It follows that
$$
        \mathbb H^*\cong \mathbb R_{>0}\times S^3.
$$
The first factor $\mathbb R_{>0}$ is central, while second factor $S^3$ is isomorphic to $\operatorname{SU}(2)$, the group of
$2\times 2$ unitary complex matrices of determinant $1$. Since
$\operatorname{SU}(2)$ is perfect, so is $S^3$, and therefore
$$
        [\mathbb H^*,\mathbb H^*]=S^3.
$$

Let $N\trianglelefteq\mathbb H^*$. If $N\subseteq\mathbb R^*$, then $N$ is
abelian. Suppose that $N\not\subseteq\mathbb R^*$. It is clear that
$$
        \mathbb H^*/\mathbb R^*
        \cong S^3/\{\pm1\}
        \cong \operatorname{SO}(3),
$$
where $\operatorname{SO}(3)$ is the group of rotations of $\mathbb R^3$.
The image of $N$ in this quotient is a nontrivial normal subgroup. Since
$\operatorname{SO}(3)$ is simple, this image is all of $\operatorname{SO}(3)$.
Thus
$$
        \mathbb H^*=N\mathbb R^*.
$$
As $\mathbb R^*$ is central, it follows that
$$
        [\mathbb H^*,\mathbb H^*]
        =[N\mathbb R^*,N\mathbb R^*]\subseteq N.
$$
Consequently $S^3\subseteq N$, and so $\mathbb H^*/N$ is abelian.

Therefore every normal subgroup $N$ of $\mathbb H^*$ has the following
property: either $N$ is abelian, or $\mathbb H^*/N$ is abelian. Hence
$\mathbb H^*$ does not have property $(\mathcal P)$.
\end{proof}

Recall that a subgroup $H$ of a group $G$ is malnormal if
$H\cap gHg^{-1}=1$ for every $g\in G\setminus H$.

\begin{proposition}
\label{prop:malnormal-obstruction}
Let $D$ be a noncommutative division ring finite-dimensional over its center
$K$. Suppose that $A$ is an abelian maximal subgroup of $D^*$, and set
$L=A\cup\{0\}$. Then $L$ is a maximal subfield of $D$, and $L^*/K^*$ is
malnormal in $D^*/K^*$. Equivalently,
$$
L^*\cap xL^*x^{-1}=K^*\quad\text{for every }x\in D^*\setminus L^*.
$$
In particular, $N_{D^*}(L^*)=L^*$.
\end{proposition}

\begin{proof}
First, we claim that $K^*\subseteq A$. Indeed, as $K^*$ is central, we get that $AK^*$ is an abelian subgroup of $D^*$
containing $A$. By the maximality of $A$, either $AK^*=A$ or $AK^*=D^*$. The second
case would make $D^*$ abelian, a contradiction. Thus $K^*\subseteq A$.

Consider $C_{D^*}(A)$. Since $A$ is abelian, $A\subseteq C_{D^*}(A)$; and so by
maximality, either $C_{D^*}(A)=A$ or $C_{D^*}(A)=D^*$. The second case
would give $A\subseteq K^*$, and hence $A=K^*$. But if $d\in D^*\setminus K^*$,
then the subgroup $\langle K^*, d\rangle$ generated by $K^*$ and $d$ is abelian and properly contains
$K^*$. It follows that $\langle K^*, d\rangle=D^*$, which yeilds $D^*$
to be abelian, a contradiction. Hence $C_{D^*}(A)=A$. It follows that $C_D(A)=A\cup\{0\}=L$. Thus $L$ is a division subring of
$D$. Since $A$ is abelian, $L$ is a field. Moreover, $C_D(L)=C_D(A)=L$, from which it follows that $L$ is a maximal subfield of $D$.

It remains to prove malnormality. Let $x\in D^*\setminus L^*$, and suppose
that $L^*\cap xL^*x^{-1}$ contains an element $a\notin K^*$. Then
$a\in L\setminus K$, and both $L$ and $xLx^{-1}$ are contained in $C_D(a)$.
Since $a\notin K$, the division ring $C_D(a)$ is a proper division subring of
$D$. Therefore $C_D(a)^*$ is a proper subgroup of $D^*$ containing $A=L^*$.
By maximality of $A$, we get $C_D(a)^*=L^*$, so $C_D(a)=L$. Thus
$xLx^{-1}\subseteq L$, and equality follows by comparing dimensions over $K$.

Hence $x$ normalizes $L^*$. Since $x\notin L^*$, the normalizer
$N_{D^*}(L^*)$ properly contains $A$. By maximality of $A$, we get
$N_{D^*}(L^*)=D^*$. The Cartan--Brauer--Hua theorem now implies that
$L\subseteq K$ or $L=D$; see \cite[(13.7)]{L}. Both alternatives are impossible: $L$ properly
contains $K$, while $D$ is noncommutative. Therefore
$L^*\cap xL^*x^{-1}=K^*$ for every $x\in D^*\setminus L^*$.

The final assertion follows at once. If $x\in N_{D^*}(L^*)$, then
$L^*\cap xL^*x^{-1}=L^*$, so malnormality forces $x\in L^*$.
\end{proof}

Recall that a field $F$ is real-closed if it is formally real and has no
formally real proper algebraic extension. Equivalently,
$\sqrt{-1}\notin F$ and $F(\sqrt{-1})$ is algebraically closed; see \cite[p.~209]{L}. In this case
$F=F^2\cup(-F^2)$. For a real-closed field $F$, Frobenius' theorem (\cite[Theorem 15.10]{L}) identifies the unique noncommutative finite-dimensional division algebra over $F$ with the quaternion division algebra
$$
        Q=(-1,-1)_F
        =F+Fi+Fj+Fk,
        \quad\text{where}\quad i^2=j^2=-1\quad\text{and}\quad ij=-ji=k.
$$
This is also the class appearing in Gerstenhaber--Yang theorem (see \cite[Theorem 15.10]{L}): if a
noncommutative division ring contains a real-closed field $R$ and is
finite-dimensional as a right $R$-space, then its center is real-closed and
the division ring is quaternion over its center.

\begin{proposition}
\label{prop:real-closed-quaternion}
Let $F$ be a real-closed field, and let $Q$ be the quaternion division algebra
over $F$. Then $Q^*$ contains no abelian maximal subgroup.
\end{proposition}

\begin{proof}
Suppose that $A$ is an abelian maximal subgroup of $Q^*$, and put
$L=A\cup\{0\}$. By Proposition~\ref{prop:malnormal-obstruction}, we conclude that $L$ is a maximal subfield of $Q$, and $N_{Q^*}(L^*)=L^*$. Since $Q$ has degree $2$ over $F$, we have $[L:F]=2$. As $F$ is real-closed,
$L/F$ is the quadratic Galois extension $F(\sqrt{-1})/F$, from which it follows that  $\operatorname{Aut}_F(L)\ne 1$.

By Skolem--Noether Theorem, the nontrivial $F$-automorphism of $L$ is induced by
conjugation by some element $q\in Q^*$. Thus $qLq^{-1}=L$. Since the induced
automorphism of $L$ is nontrivial, $q\notin L^*$. Hence
$q\in N_{Q^*}(L^*)\setminus L^*$, contradicting $N_{Q^*}(L^*)=L^*$.
Therefore $Q^*$ has no abelian maximal subgroup.
\end{proof}

\begin{corollary}
\label{cor:H-no-abelian-maximal}
The division ring $\mathbb H$ of real quaternions has no abelian maximal subgroup in $\mathbb H^*$.
\end{corollary}

\begin{proof}
The division ring $\mathbb H$ is the quaternion division algebra over the
real-closed field $\mathbb R$. The result follows from
Proposition~\ref{prop:real-closed-quaternion}.
\end{proof}

\begin{theorem}
\label{thm:real-closed-subfield-no-abelian-maximal}
Let $D$ be a noncommutative division ring containing a real-closed field $R$
such that $\dim(D_R)<\infty$. Then $D^*$ contains no abelian maximal subgroup.
\end{theorem}

\begin{proof}
Let $F=Z(D)$. By the Gerstenhaber--Yang theorem, $F$ is real-closed and $D$
is the quaternion division algebra over $F$; see \cite[(15.10)]{L}.
Proposition~\ref{prop:real-closed-quaternion} applies to $D$, and therefore
$D^*$ contains no abelian maximal subgroup.
\end{proof}

{\noindent\textbf{Funding.} This work is funded by Vietnam National Foundation
for Science and Technology Development (NAFOSTED) under Grant No.
101.04-2025.41.}

\end{document}